\documentclass[reqno, 12pt]{amsart}
\pdfoutput=1
\makeatletter
\let\origsection=\section \def\section{\@ifstar{\origsection*}{\mysection}}
\def\mysection{\@startsection{section}{1}\z@{.7\linespacing\@plus\linespacing}{.5\linespacing}{\normalfont\scshape\centering\S}}
\makeatother

\usepackage{amsmath,amssymb,amsthm}
\usepackage{mathrsfs}
\usepackage{mathabx}\changenotsign
\usepackage{dsfont}

\usepackage[normalem]{ulem}

\usepackage[dvipsnames]{xcolor}
\usepackage[backref]{hyperref}
\hypersetup{
    colorlinks,
    linkcolor={red!60!black},
    citecolor={green!60!black},
    urlcolor={blue!60!black}
}

\usepackage{graphicx}

\usepackage[open,openlevel=2,atend]{bookmark}

\usepackage[abbrev,msc-links,backrefs]{amsrefs}
\usepackage{doi}

\renewcommand{\PrintDOI}[1]{\doi{#1}}

\usepackage[T1]{fontenc}
\usepackage{lmodern}
\usepackage[babel]{microtype}
\usepackage[english]{babel}

\usepackage{geometry}
\numberwithin{equation}{section}
\numberwithin{figure}{section}

\usepackage{enumitem}
\def\rmlabel{\upshape({\itshape \roman*\,})}

\usepackage{pgf,tikz}

\makeatletter
\def\greek#1{\expandafter\@greek\csname c@#1\endcsname}
\def\Greek#1{\expandafter\@Greek\csname c@#1\endcsname}
\def\@greek#1{\ifcase#1
	\or $\alpha$%
	\or $\beta$%
	\or $\gamma$%
	\or $\delta$%
	\or $\epsilon$%
	\or $\zeta$%
	\or $\eta$%
	\or $\theta$%
	\or $\iota$%
	\or $\kappa$%
	\or $\lambda$%
	\or $\mu$%
	\or $\nu$%
	\or $\xi$%
	\or $o$%
	\or $\pi$%
	\or $\rho$%
	\or $\sigma$%
	\or $\tau$%
	\or $\upsilon$%
	\or $\phi$%
	\or $\chi$%
	\or $\psi$%
	\or $\omega$%
\fi}
\def\@Greek#1{\ifcase#1
	\or $\mathrm{A}$%
	\or $\mathrm{B}$%
	\or $\Gamma$%
	\or $\Delta$%
	\or $\mathrm{E}$%
	\or $\mathrm{Z}$%
	\or $\mathrm{H}$%
	\or $\Theta$%
	\or $\mathrm{I}$%
	\or $\mathrm{K}$%
	\or $\Lambda$%
	\or $\mathrm{M}$%
	\or $\mathrm{N}$%
	\or $\Xi$%
	\or $\mathrm{O}$%
	\or $\Pi$%
	\or $\mathrm{P}$%
	\or $\Sigma$%
	\or $\mathrm{T}$%
	\or $\mathrm{Y}$%
	\or $\Phi$%
	\or $\mathrm{X}$%
	\or $\Psi$%
	\or $\Omega$%
\fi}

\AddEnumerateCounter{\greek}{\@greek}{24}
\AddEnumerateCounter{\Greek}{\@Greek}{12}

\makeatother

\let\polishlcross=\l
\def\l{\ifmmode\ell\else\polishlcross\fi}

\def\paragraph#1{%
  \noindent\textbf{#1.}\enspace}

\let\emptyset=\varnothing
\let\setminus=\smallsetminus

\makeatletter
\def\moverlay{\mathpalette\mov@rlay}
\def\mov@rlay#1#2{\leavevmode\vtop{   \baselineskip\z@skip \lineskiplimit-\maxdimen
   \ialign{\hfil$\m@th#1##$\hfil\cr#2\crcr}}}
\newcommand{\charfusion}[3][\mathord]{
    #1{\ifx#1\mathop\vphantom{#2}\fi
        \mathpalette\mov@rlay{#2\cr#3}
      }
    \ifx#1\mathop\expandafter\displaylimits\fi}
\makeatother

\DeclareFontFamily{U}  {MnSymbolC}{}
\DeclareSymbolFont{MnSyC}         {U}  {MnSymbolC}{m}{n}
\DeclareFontShape{U}{MnSymbolC}{m}{n}{
    <-6>  MnSymbolC5
   <6-7>  MnSymbolC6
   <7-8>  MnSymbolC7
   <8-9>  MnSymbolC8
   <9-10> MnSymbolC9
  <10-12> MnSymbolC10
  <12->   MnSymbolC12}{}
\DeclareMathSymbol{\powerset}{\mathord}{MnSyC}{180}

\usepackage{tikz}
\usetikzlibrary{calc,decorations.pathmorphing}
\pgfdeclarelayer{background}
\pgfdeclarelayer{foreground}
\pgfdeclarelayer{front}
\pgfsetlayers{background,main,foreground,front}

\let\epsilon=\varepsilon
\let\eps=\epsilon
\let\rho=\varrho
\let\theta=\vartheta
\let\kappa=\varkappa

\let\E=\EE

\def\PP{{\mathds P}}

\def\rm{\mathbb{RM}}

\theoremstyle{plain}
\newtheorem{thm}{Theorem}[section]
\newtheorem{theorem}[thm]{Theorem}

\newtheorem{prop}[thm]{Proposition}

\newtheorem{cor}[thm]{Corollary}

\newtheorem{lemma}[thm]{Lemma}

\theoremstyle{definition}

\newtheorem{conj}[thm]{Conjecture}

\usepackage{accents}

\let\phi=\varphi

\begin{document}

\title[Twins in ordered hyper-matchings]{Twins in ordered hyper-matchings}

\author{Andrzej Dudek}
\address{Department of Mathematics, Western Michigan University, Kalamazoo, MI, USA}
\email{\tt andrzej.dudek@wmich.edu}
\thanks{The first author was supported in part by Simons Foundation Grant \#522400.}

\author{Jaros\l aw Grytczuk}
\address{Faculty of Mathematics and Information Science, Warsaw University of Technology, Warsaw, Poland}
\email{jaroslaw.grytczuk@pw.edu.pl}
\thanks{The second author was supported in part by Narodowe Centrum Nauki, grant 2020/37/B/ST1/03298.}

\author{Andrzej Ruci\'nski}
\address{Department of Discrete Mathematics, Adam Mickiewicz University, Pozna\'n, Poland}
\email{\tt rucinski@amu.edu.pl}
\thanks{The third author was supported in part by Narodowe Centrum Nauki, grant 2018/29/B/ST1/00426}

\begin{abstract}An \emph{ordered $r$-matching} of \emph{size} $n$ is an $r$-uniform hypergraph on a linearly ordered set of vertices, consisting of $n$ pairwise disjoint edges. Two ordered $r$-matchings are \emph{isomorphic} if there is an order-preserving isomorphism between them. A pair of \emph{twins} in an ordered $r$-matching is formed by two vertex disjoint isomorphic sub-matchings. Let $t^{(r)}(n)$ denote the maximum size of twins one may find in \emph{every} ordered $r$-matching of size $n$.
	
By relating the problem to that of largest twins in permutations and applying some recent Erd\H os-Szekeres-type results for ordered matchings, we show that  $t^{(r)}(n)=\Omega\left(n^{\frac{3}{5\cdot(2^{r-1}-1)}}\right)$ for every fixed $r\geqslant 2$. On the other hand, $t^{(r)}(n)=O\left(n^{\frac{2}{r+1}}\right)$, by a simple probabilistic argument. As our main result, we prove that, for \emph{almost all} ordered $r$-matchings of size $n$, the size of the largest twins  achieves this bound.
\end{abstract}

\maketitle


\setcounter{footnote}{1}

\section{Introduction}
Let $r\ge2$ be a fixed integer. An \emph{ordered $r$-matching} of \emph{size} $n$ is an $r$-uniform hypergraph $M$ on a linearly ordered vertex set~$V$, with $|V|=rn$, consisting of $n$ pairwise disjoint edges. One may represent such a matching as a \emph{word} with $n$ distinct letters in which every letter occurs exactly $r$ times. For instance, the word $AABCBDBDACCD$ represents an ordered $3$-matching of size $4$ whose edges correspond to sets of positions occupied by a given letter, i.e., $A=\{1,2,9\}$, $B=\{3,5,7\}$, $C=\{4,10,11\}$, and $D=\{6,8,12\}$. We will frequently identify an ordered matching with its representing word.

Two ordered matchings are \emph{isomorphic} if there is an order-preserving bijection between their vertex sets inducing a one-to-one correspondence between their edge sets. In particular, the words representing isomorphic ordered matchings are identical up to renaming the letters. A pair of two vertex disjoint isomorphic sub-matchings of an ordered matching $M$ is called \emph{twins}. By the \emph{size} of twins we mean the size of just one of the two in the pair. The maximum size of twins in a matching $M$ will be denoted by $t(M)$. The corresponding extremal function is defined by $$t^{(r)}(n)=\min\{t(M):M\text{ is an ordered $r$-matching of size $n$}\}.$$

For instance, in the matching $M$ represented by the word $$\color{red}AA\color{blue}B\color{black}E\color{red}C\color{blue}BD\color{black}EE\color{blue}BD\color{red}ACC\color{blue}D$$
one can find a pair of twins of size two formed by the sub-matchings $\color{red}AACACC$ and  $\color{blue}BBDBDD$. So, $t(M)=2$ as, trivially, we always have $t(M)\le n/2$ for a matching of size $n$.

In a recent paper \cite{DGR-Latin} (see also \cite{DGR-match}) we demonstrated that the following inequalities hold for all  $n\ge2$ (note that $t^{(2)}(1)=0$): $$\frac{1}{16\sqrt[5]{2}}\cdot n^{\frac{3}{5}}\leqslant t^{(2)}(n)\leqslant \frac{e}{\sqrt[3]{2}}\cdot n^{\frac23}.$$
The upper bound is obtained by a standard probabilistic argument based on the expectation. In fact, we proved in \cite{DGR-Latin} that this upper bound is a.a.s.\ (asymptotically almost surely) attained by almost all ordered matching of size $n$. The lower bound is a consequence of a result by Bukh and Rudenko \cite{BukhR}, concerning the related problem  for permutations.

By \emph{twins} in a permutation $\pi$ we mean a pair of disjoint order-isomorphic subsequences of~$\pi$. Let $\tau(n)$ be the maximum length of twins contained in every permutation of length~$n$. The result in \cite{BukhR} states that $\tau(n)\ge\tfrac18\cdot n^{\frac{3}{5}}$, which is so far the best lower bound, while $\tau(n)=O\left(n^{\frac{2}{3}}\right)$ follows by an elementary probabilistic argument. It is conjectured by Gawron \cite{Gawron} that this upper bound yields the correct order of magnitude for $\tau(n)$. If true, this would imply the same for the function $t^{(2)}(n)$, since we proved in \cite{DGR-Latin, DGR-match} that $t^{(2)}(n)=\Theta(\tau(n))$. For other related results on twins in permutations, as well as for the rich background of the problem, we encourage the reader to look at our paper \cite{DGR}.

In the present paper we extend the above results to ordered $r$-matchings, $r\geqslant 3$, in both, the deterministic and the random setting. Let $\rm_{n} := \rm^{(r)}_{n}$ denote a \emph{random} ordered $r$-matching of size $n$ (for precise definition, see the next section).

\begin{thm}\label{thm:random}
	For every $r\ge2$, a.a.s.,
	\begin{equation}\label{eq:thm:random}
			t(\rm_n) = \Theta\left(n^{\frac{2}{r+1}}\right).
	\end{equation}	
\end{thm}

The proof of the upper bound follows, once again, by a standard application of the first moment method. To get the lower bound we apply a more sophisticated argument, similar to that in \cite{DGR} and \cite{DGR-Latin,DGR-match}, which is based on  a concentration inequality of Talagrand for permutations \cite{LuczakMcDiarmid} (see Section~\ref{rm}).

Of course, the upper bound in (\ref{eq:thm:random}) is also true in the deterministic case, so, we have $t^{(r)}(n)=O\left(n^{\frac{2}{r+1}}\right)$, for every $r\geqslant 2$. Our second result gives a lower bound for $t^{(r)}(n)$.

\begin{thm}\label{thm:deterministic}
	For every $r\geqslant 2$, we have
	\begin{equation}\label{eq:thm:deterministic}
		t^{(r)}(n) = \Omega\left(n^{\frac{3}{5}\cdot \frac{1}{2^{r-1}-1}}\right).
	\end{equation}
\end{thm}
The proof of Theorem \ref{thm:deterministic} in Section \ref{gen}, is based on  a key lemma (Lemma \ref{double_ind}), establishing a  double recurrence for $t^{(r)}$,  and a recent Erd\H{o}s-Szekeres type result~\cite{ErdosSzekeres} for ordered matchings obtained by  Sauerman and Zakharov \cite{SZ}.
 This latter result states that every sufficiently large ordered $r$-matching $M$ contains a large homogenous ``clique'', i.e., a sub-matching $M'$ whose all pairs of edges form pairwise  isomorphic $r$-matchings (of size two).
 More specifically, the size of such a ``clique'' $M'$ in any ordered $r$-matching with $n$ edges is at least $\frac{1}{2}\cdot n^{\frac{1}{2^r-1}}$, as proved in \cite{SZ}.

 Observe that if $M'=\{e_1,\dots,e_m\}$, then by arbitrarily splitting $M'$ in half, e.g., $\{e_1,\dots,e_{\lfloor m/2\rfloor}\}$ and $\{e_{\lceil m/2\rceil+1},\dots,e_m\}$, we obtain twins. Thus, we immediately get the lower bound $t^{(r)}(n) = \Omega\left(n^{\frac{1}{2^r-1}}\right)$, which is, however, much worse than the bound in (\ref{eq:thm:deterministic}). For instance, the exponents of $n$ in these two bounds are equal, respectively, to $\frac15$ and $\frac17$ ($r=3$), to $\frac{3}{35}$ and $\frac{1}{15}$ ($r=4$), and to $\frac{1}{25}$ and $\frac{1}{31}$ ($r=5$). In fact, their ratio converges to $\frac{6}{5}$ as $r$ is growing. Nonetheless, the ratio of the logarithms of the current upper and lower bounds for $t^{(r)}(n)$ grows rather rapidly with $r$.

We will use the standard notation $[n]:=\{1,2,\ldots,n\}$.

\section{Random matchings}\label{rm}

In this section we prove Theorem \ref{thm:random}. Recall that $\rm_{n} := \rm^{(r)}_{n}$ denotes a random ordered $r$-matching of size $n$, that is, an ordered $r$-matching picked uniformly at random out of all
\[
\alpha^{(r)}_n:=\frac{(rn)!}{(r!)^n\, n!}
\]
such $r$-matchings on the set $[rn]$.

 The formula for $\alpha^{(r)}_n$ indicates that each ordered matching can be coupled with exactly $(r!)^n n!$ permutations. Indeed, one can generate an ordered matching by the following permutational scheme. Let $\pi$ be a permutation of $[rn]$. Now $\pi$ can be chopped off into an $r$-matching consisting of the following collection of edges
 $$\{\pi(1),\dots,\pi(r)\}, \{\pi(r+1),\dots,\pi(2r)\},\dots,\{\pi(rn-r+1),\dots,\pi(rn)\}.$$
 Clearly, there are exactly $(r!)^n n!$ permutations $\pi$ yielding the same matching. Thus, a (uniformly) random permutation $\Pi_{rn}$ of $[rn]$ generates a (uniformly) random $r$-matching $\rm_n$. This scheme  allows one to use concentration inequalities (such as the Talagrand inequality~\cite{Talagrand}) for random permutations in the context of random matchings.

\begin{proof}[Proof of Theorem \ref{thm:random}: upper bound] The upper bound in formula (\ref{eq:thm:random}) was already mentioned without proof in~\cite[Section 2]{DGR-rM}. The proof is based on the standard first moment method. For each $k$, let $X_k$ be
the number of twins of size $k$ in $\rm_{n}$. Then
\[
\E X_k=\frac1{2!}\cdot\binom{rn}{rk,rk,rn-2rk}\cdot\frac{\alpha^{(r)}_k\cdot 1 \cdot \alpha^{(r)}_{n-2k}}{\alpha^{(r)}_n}
= \frac{1}{2} \cdot \frac{n!}{(n-2k)!} \cdot \frac{1}{k!} \cdot (r!)^{k} \cdot \frac{1}{(rk)!},
\]
where the factor of 1 represents the second twin which is fully determined by the first one.

  Ignoring $\tfrac12$ and using inequalities $n! / (n-2k)! \le n^{2k}$, $k! \ge (k/e)^{k}$, $r! \le r^r$, and $(rk)! \ge (rk/e)^{rk}$, we thus get
\[
\E X_k\le n^{2k} \cdot \left(\frac{e}{k}\right)^k \cdot r^{rk} \cdot \left( \frac{e}{rk} \right)^{rk}= \left(n^{2} \cdot \frac{e}{k} \cdot r^{r} \cdot \left( \frac{e}{rk} \right)^{r}\right)^k
= \left( \frac{n^2}{\left(k/e \right)^{r+1}} \right)^k,
\]
which converges to 0 as $n\to\infty$ for $k\ge cn^{2/(r+1)}$, with any constant $c>e$.
Hence, for such~$k$, $\PP(X_k>0)\le\E X_k=o(1)$ and so, a.a.s.\ $t^{(r)}(\rm_n) \le cn^{2/(r+1)}$. \end{proof}

\medskip

An important ingredient of the proof of the lower bound in Theorem~\ref{thm:random} is a Talagrand's concentration inequality for random permutations from \cite{Talagrand}. We quote here a slightly simplified  version from \cite[Inequality (2) with $l=2$]{LuczakMcDiarmid} (see also \cite[Inequlity (1.3)]{McDiarmid}). Let ${\Pi}_{n'}$ be a random permutation of order $n'$ (we will be applying this theorem with $n'=rn$).

\begin{theorem}[Luczak and McDiarmid \cite{LuczakMcDiarmid}]\label{tala}
	Let $h(\pi)$ be a function defined on the set of all permutations of order $n'$ which, for some positive constants $c$ and $d$, satisfies
	\begin{enumerate}[label=\rmlabel]
		\item\label{thm:talagrand:i} if $\pi_2$ is obtained from $\pi_1$ by swapping two elements, then $|h(\pi_1)-h(\pi_2)|\le c$;
		\item\label{thm:talagrand:ii} for each $\pi$ and $s>0$, if $h(\pi)=s$, then in order to show that $h(\pi)\ge s$, one needs to specify only at most $ds$ values $\pi(i)$.
	\end{enumerate}
	Then, for every $\eps>0$,
	$$\PP(|h({\Pi}_{n'})-m|\ge \eps m)\le4 \exp(-\eps^2 m/(32dc^2)),$$
	where $m$ is the median of the random variable $h({\Pi}_{n'})$.
\end{theorem}
 \noindent As mentioned above, one can use this lemma for random $r$-matchings, as they can be generated by random permutations.

\medskip

\begin{proof}[Proof of Theorem \ref{thm:random}: lower bound]
Set
\[
a:=\beta n^{(r-1)/(r+1)} \quad \text{ with } \quad \beta:=\frac{1}{(20er!)^{1/(r+1)}}
\]
and
\[
N:=\frac{rn}{a}=\frac{r}{\beta} n^{2/(r+1)}.
\]
For simplicity we assume that both $a$ and $N$ are integers. Partition  $[rn]=A_1\cup\cdots\cup A_{N}$, where $A_i$'s are consecutive blocks of  $a$ integers each. For every $I=\{i_1,\dots,i_r\}$ with $1\le i_1<\dots<i_r\le N$, we call an $r$-element subset $S\subset [rn]$ an \emph{$I$-set} if $|S\cap A_{i_j}|=1$ for each $1\le j\le r$.
Further, define a random variable $X_{I}$ which counts the number of edges of $\rm_n$ which are $I$-sets.

Consider an auxiliary $r$-uniform hypergraph $H:=H(\rm_n)$ on vertex set $[N]$, where $I$ is an edge in $H$ if and only if $X_{I}\ge 2$. For better clarity the edges of $H$ will be sometimes called $H$-edges. Trivially, the maximum degree in $H$ is at most $\binom{N-1}{r-1}$, but also, due to the disjointness of edges in $\rm_n$, at most $a/2$.

It is easy to see that a matching of size $k$ in $H$ corresponds to twins in~$\rm_n$ of size~$k$. Indeed,  let $M=\{I_1,\dots,I_k\}$ be a matching in $H$. For every $1\le \ell\le k$, let $e^{(\ell)}_1,e^{(\ell)}_2\in \rm_n$ be two $I_\ell$-edges. Then, the sub-matchings $M_1=\{e^{(1)}_1,\dots,e^{(k)}_1\}$ and $M_2=\{e^{(1)}_2,\dots,e^{(k)}_2\}$, owing to the sequential choice of $A_i$'s, form twins in $\rm_n$.
Thus, our ultimate goal is to show that a.a.s.~$H$ contains a matching of size $\Omega(n^{2/(r+1)})$. Let $\nu(H)$ be the size of the largest matching in $H$. Our ultimate goal is thus to show that a.a.s.\ $\nu(H)=\Omega(N)$. Anticipating  application of Theorem \ref{tala} to $\nu(H)$, it will be sufficient to  bound $\E(\nu(H))$ from below.
  We will do it in a most ``silly'' way by looking just for isolated $H$-edges.

 Let $H_1$ be a subgraph of $H$ induced by the set $V_1$ of vertices of degrees at most $1$ in~$H$, that is, $E(H_1)$ consists of isolated edges in $H$ which, of course, form a matching.  Set $W=|E(H)|$ and $W_1=|E(H_1)|$ for the random variables counting the edges in $H$ and $H_1$, respectively. Then,  $\E(\nu(H))\ge\E(W_1)$ and
\[
W_1 = W-|\{e\in E(H):\; e\cap (V\setminus V_1)\neq\emptyset\}|
\ge W-\sum_{d= 2}^{\lfloor a/2\rfloor} dZ_d,
\]
 where $Z_d$ counts the number of vertices of degree $d$ in $H$.  Note that $\E(Z_d)=N\PP(D=d)$, where $D$ is the degree of a fixed vertex, say vertex 1, in $H$. Thus, we have

 \begin{equation}\label{EWW1}
\E(W_1) \ge \E(W) - \E\left(\sum_{d=2}^{\lfloor a/2\rfloor} dZ_d\right)=\E(W)-N\sum_{d=2}^{\lfloor a/2\rfloor}d\PP(D=d)
 \end{equation}
\noindent and all we need are a lower bound on $\E(W)$ and an upper bound on $\PP(D=d)$, so that \eqref{EWW1} would yield $\E(W_1)=\Omega(N)$.
We begin with the former task. From the definition of $H$ we have

\begin{equation}\label{eq:expectatio_EH}
\E(W) = \sum_{I\in \binom{[N]}{r}} \PP(X_{I}\ge 2)
\end{equation}
which, however, unlike in \cite{DGR-match}, cannot be applied directly. The reason is that to bound $\PP(X_{I}\ge 2)$ from below one has to handle
expressions like $n!/(n-a)!$ which are asymptotic to $n^a$ as long as $a=o(\sqrt n)$. In \cite{DGR-match} the case of graphs ($r=2$) was considered
and we had $a=\Theta(n^{1/3})$, however, for $r\ge 3$ we have $a=\Theta(n^{(r-1)/(r+1)})=\Omega(\sqrt n)$.

Hence, instead of applying \eqref{eq:expectatio_EH} directly,  we define a random variable $Y$ that counts the number of pairs of edges $\{e_1,e_2\}$ for which there exists a set of indices $I=\{ i_1<\dots<i_r\}$ such that  both $e_1$ and $e_2$ are $I$-edges. By definitions, $Y$ and $X_I$ are related by the identity
\[
Y = \sum_{I \in \binom{[N]}{r},\; X_I\ge 2} \binom{X_{I}}{2}.
\]
Observe that for any  integer $t\ge 2$, using the obvious bound $X_{I} \le a$, we get

\[
Y = \sum_{I \in \binom{[N]}{r},\; 2\le X_I\le t} \binom{X_{I}}{2}
+ \sum_{I \in \binom{[N]}{r},\; X_I\ge t+1} \binom{X_{I}}{2}\le\binom{t}{2} \sum_{I \in \binom{[N]}{r},\; X_I\ge 2} 1
+\binom{a}{2} \sum_{I \in \binom{[N]}{r},\; X_I\ge t+1} 1.
\]

Thus,
\begin{align}
\E(Y) &\le \binom{t}{2} \sum_{I \in \binom{[N]}{r}} \PP(X_I\ge 2)
+\binom{a}{2} \sum_{I \in \binom{[N]}{r}} \PP(X_I\ge t+1) \notag\\
&= \binom{t}{2} \E(W)
+\binom{a}{2} \sum_{I \in \binom{[N]}{r}} \PP(X_I\ge t+1),\label{eq:Y2}
\end{align}
where we used~\eqref{eq:expectatio_EH}.

From here to obtain a lower bound on $\E(W)$ it suffices to bound $\E(Y)$ from below and show that, for some constant $t$, the latter summation above is $o(\E(Y))$. Observe that $\E(Y)$ can be calculated directly from its definition as
\[
\E(Y) = \binom{N}{r} \binom{a}{2}^r (2!)^{r-1} \cdot \frac{\alpha^{(r)}_{n-2}}{\alpha^{(r)}_n}
= \binom{N}{r} \binom{a}{2}^r (2!)^{r-1} \cdot (r!)^2 \frac{n!}{(n-2)!} \frac{(rn-2r)!}{(rn)!}.
\]
Indeed, there are $\binom Nr$ choices of $I=\{i_1,\dots,i_r\}$ and, given that, there are $\binom{a}{2}^r (2!)^{r-1}$ choices of two $I$-sets $e_1,e_2$. The quotient $\tfrac{\alpha^{(r)}_{n-2}}{\alpha^{(r)}_n}$ equals the probability that $\rm_n$ contains $e_1$ and $e_2$ as edges and the formula follows by the linearity of expectation.

Since $r$ is fixed (and the asymptotic is taken in $n$), each binomial coefficient can be easily approximated and furthermore $\frac{n!}{(n-2)!}\sim n^{2}$ and $\frac{(rn-2r)!}{(rn)!} \sim \frac{1}{(rn)^{2r}}$. This yields,
\[
\E(Y)
\sim \frac{N^r}{r!} \left( \frac{a^2}{2!} \right)^r (2!)^{r-1} \cdot \left( \frac{r! n}{(rn)^r} \right)^2
= \frac{N^r}{2r!} \left( \frac{(r-1)! a^r}{(rn)^{r-1}} \right)^2.
\]
Now expressing $N$ and $a$ as functions of $n$ (leaving just one factor of $N$ intact), we get
\begin{align*}
\E(Y) \sim
N \cdot \frac{\left(\frac{r}{\beta} n^{2/(r+1)}\right)^{r-1}}{2r!} \left( \frac{(r-1)! \beta^r n^{(r-1)r/(r+1)}}{(rn)^{r-1}} \right)^2= N \cdot \frac{(r-1)!\beta^{r+1}}{2r^r}
\end{align*}
and thus, for large $n$,
\begin{equation}\label{eq:expectationY2}
\E(Y) \ge  N \cdot \frac{(r-1)!\beta^{r+1}}{3r^r}.
\end{equation}

Next, we estimate the second summation in~\eqref{eq:Y2}. Note that for a fixed $I\in \binom{[N]}{r}$,
\[
\PP(X_I\ge t+1) \le \binom{a}{t+1}^r ((t+1)!)^{r-1} \frac{\alpha^{(r)}_{n-(t+1)}}{\alpha^{(r)}_n}
\sim \frac{a^{r(t+1)}}{(t+1)!} \frac{ (r!)^{t+1} n^{t+1}}{(rn)^{r(t+1)}}.
\]

Hence,
\[
\PP(X_I\ge t+1)=O_{r,t}\left(n^{-\frac{(r-1)(t+1)}{r+1}} \right)
\]
and consequently,
\[
\binom{a}{2} \sum_{I \in \binom{[N]}{r}} \PP(X_I\ge t+1)
= O_{r,t}\left(a^2 N^r n^{-\frac{(r-1)(t+1)}{r+1}} \right)
= O_{r,t}\left(n^{\frac{2(r-1)+2r-(r-1)(t+1)}{r+1}} \right).
\]
Since for any $r\ge 2$ and $t\ge4$,
\[
\frac{2(r-1)+2r-(r-1)(t+1)}{r+1} = \frac{(r-1)(3-t) + 2}{r+1} < \frac{2}{r+1},
\]
 we obtain, taking $t=4$,
\begin{equation}\label{eq:sumXt}
\binom{a}{2} \sum_{I \in \binom{[N]}{r}} \PP(X_I\ge 5) = o(N).
\end{equation}
Using \eqref{eq:expectationY2} and~\eqref{eq:sumXt} in~\eqref{eq:Y2}, this implies the bounds
\[
N \cdot \frac{(r-1)!\beta^{r+1}}{3r^r}\le \E(Y) \le \binom{4}{2} \E(W) + o(N)
\]
from which it follows that
\[
\E(W) \ge
N \cdot \frac{(r-1)!\beta^{r+1}}{20r^r}.
\]

After having estimated $\E(W)$ we move to the second major task which is to bound
$\PP(D=d)$ from above where, recall, $D$ is the degree of vertex 1 in $H$. 
Every edge of $H$ containing vertex 1 corresponds  to a set $I\in \binom{[N]}r$ with $1\in I$ and (not uniquely) to a pair of to-be  $I$-edges $e_1,e_2$ in $\rm_n$. The number of choices of the triple $(I,e_1,e_2)$ is $\binom{N-1}{r-1}\binom{a}{2}^{r} (2!)^{r-1}$ and the same bound applies to the $d-1$ remaining triples $(I',e_1,e_2')$. Of course, the edges selected to the triples should be vertex-disjoint as otherwise they could not all appear in a matching. As we bound from above, we may ignore this requirement. On the other hand, as the $H$-edges containing vertex 1 are not ordered, we need to divide by $d!$.

Consequently, very crudely,
\begin{align*}
\PP(D=d)\le\PP(D\ge d)
&\le \frac{1}{d!}\left( \binom{N}{r-1}\binom{a}{2}^{r} (2!)^{r-1} \right)^d \frac{\alpha^{(r)}_{n-2d}}{\alpha^{(r)}_n}\\
&\sim \frac{1}{d!} \left( \frac{N^{r-1}}{(r-1)!} \frac{a^{2r}}{2} \right)^d \frac{(rn-2rd)!}{(rn)!} \frac{n!}{(n-2d)!} (r!)^{2d}.
\end{align*}
\noindent Above we turned to estimating $\PP(D\ge d)$ instead of $\PP(D= d)$ to avoid the issue of producing incidentally more than $d$ edges of $H$ containing vertex 1.

Using the inequality $1-x\ge e^{-2x}$ valid for $x\le 1/2$, the fraction $\frac{(rn-2rd)!}{(rn)!}$ can be estimated as
\[
\frac{(rn-2rd)!}{(rn)!}=\frac1{(rn)^{2rd}\left(1-\frac1{rn}\right)\cdot\ldots\cdot\left(1-\frac{2rd-1}{rn}\right)}\le\frac{e^{4rd^2/n}}{(rn)^{2rd}}
\]
and the fraction $\frac{n!}{(n-2d)!}$ is trivially bounded by $n^{2d}$. Hence,
\[
\PP(D=d)\le \frac{1}{d!} \left( \frac{N^{r-1}}{(r-1)!} \frac{a^{2r}}{2} \right)^d \frac{e^{4rd^2/n}}{(rn)^{2rd}}\; n^{2d} (r!)^{2d}
= \frac{1}{d!} \left(\frac{N^{r-1}}{(r-1)!} \frac{a^{2r}}{2}  \frac{e^{4rd/n}}{(rn)^{2r}}\; n^{2} (r!)^{2} \right)d.
\]
Since $N^{r-1} a^{2r} \frac{n^2}{n^{2r}} = r^{r-1} \beta^{r+1}$ and $e^{4rd/n} \le 2$ as $d \le a/2=o(n)$, we get
\[
\PP(D=d)\le \frac{1}{d!} \left( \frac{1}{(r-1)!} \cdot \frac{1}{2} \cdot \frac{2}{r^{2r}} (r!)^2 \cdot r^{r-1} \beta^{r+1}\right)^d
= \frac{1}{d!} \left( \frac{\beta^{r+1} r!}{r^{r}} \right)^d.
\]

 We are now ready to apply \eqref{EWW1} and get
\[
\E(W_1) = \E(W) -N\sum_{d\ge2}d\PP(D=d)
\ge N \left(\frac{(r-1)!\beta^{r+1}}{20r^r} - \sum_{d\ge 2} \frac{1}{(d-1)!} \left( \frac{\beta^{r+1} r!}{r^{r}}  \right)^d \right).
\]
Since $\frac{\beta^{r+1} r!}{r^{r}}=\frac{1}{20er^{r}} < 1$, we can bound
\[
\sum_{d\ge 2} \frac{1}{(d-1)!} \left( \frac{\beta^{r+1} r!}{r^{r}} \right)^d
\le \left( \frac{1}{20er^{r}}\right)^2 \sum_{d\ge 2} \frac{1}{(d-1)!}
\le \left(\frac{1}{20er^{r}} \right)^2 e,
\]
and thus,
\[
E(\nu(H))\ge\E(W_1) \ge N \cdot \frac{1}{20^2er^{r+1}} \left(1 - \frac{1}{r^{r-1}}\right)
=\Omega_r(N)=\Omega_r(n^{2/(r+1)}).
\]
Finally, owing to the permutational scheme of generating $\rm_n$, we are in a position to apply Theorem \ref{tala} with $h(\pi)=\nu(H)$. Let us check the assumptions.
For a permutation $\pi$ of $[rn]$, let $M(\pi)$ be the corresponding matching.
Observe that if $\pi_2$ is obtained from a permutation $\pi_1$ by swapping some two of its elements, then at most two edges of $M(\pi_1)$ can be destroyed and at most two edges of $M(\pi_1)$ can be created, and thus the same can be said about the $H$-edges of in $M(\pi_1)$. This, in turn, implies that the size of the largest matching in $H$ has been altered by at most two, that is, $|h(\pi_1)-h(\pi_2)|\le2$.
Moreover, to exhibit that $h(\pi)\ge s$, it obviously suffices to reveal $2s$ edges of $\rm_n$, that is, $2rs$ values of~$\pi$.

 Thus, Theorem \ref{tala} with $c=2$, $d=2r$, and $\eps=1/2$ yields that
 $$\PP(|\nu(H)-m|\ge m/2)\le4 \exp(-m/(1024r)).$$
 Moreover, there is a standard passage from the median to the expectation $\mu=\E(\nu(H))$.
 Indeed, we have (see for example \cite{Talagrand}, Lemma 4.6, or \cite{McDiarmid}, page 164)  that $|m -\mu|=O(\sqrt m)$. As demonstrated above, $\mu\to\infty$, so it follows that $m\to\infty$ and, in particular,  $|m-\mu|\le 0.01\mu$. This implies that $\PP(|\nu(H)-m|\ge m/2) = o(1)$ and
	\begin{align*}
		\PP(|\nu(H)-\mu|\ge(3/4)\mu)
		&= \PP(|\nu(H)-m + m-\mu|\ge(3/4)\mu)\\
		&\le \PP(|\nu(H)-m| + |m-\mu|\ge(3/4)\mu)\\
		&\le \PP(|\nu(H)-m| \ge(2/3)\mu)
		\le \PP(|\nu(H)-m|\ge m/2) = o(1),
	\end{align*}
which means that a.a.s.\ $\nu(H)=\Omega_r(\mu)=\Omega_r(n^{2/(r+1)})$, from which the existence of twins in $\rm_n$ of size $\Omega_r(n^{2/(r+1)})$ follows. \end{proof}

Notice that we could not apply Theorem \ref{tala} directly to  the random variable $W_1$, as in order to exhibit $s$ isolated edges of $H$,  one would need to reveal an unbounded number of values of $\pi$ -- to make sure that none of the  $r$-element subsets of $[N]$ intersecting the given $s$ forms an edge of $H$. In fact, for $r=2$, one could use instead the Azuma-Hoeffding inequality (see, e.g., \cite[Theorem 3.7]{DGR-match} or the references given there), avoiding the above issue with the witness assumption of Talagrand's inequality. Unfortunately, this does not work for $r\ge3$, as then we only have $\E(W_1)=\Omega(n^{2/(r+1)})$, so we do not know if $\left(\E(W_1)\right)^2/n$  tends to infinity which makes the Azuma-Hoeffding inequality useless.

\section{General matchings}\label{gen}

In this section we will give the proof of Theorem \ref{thm:deterministic}. We start with stating the main lemma together with some explanations of how it leads to the desired lower bound (\ref{eq:thm:deterministic}).

\subsection{The main lemma and its consequence}

In \cite[Lemma 3.4, $r=2$]{DGR-match} we showed the following lower bound on $t^{(2)}(n)$ in terms of $\tau(n)$. (Note that in \cite{DGR-match}, unlike here, $r$ meant multiplicity of twins.)

\begin{prop}\label{polyn}
	For all $3/5\le\alpha\le2/3$ and $\beta>0$, if $\tau(n)\ge \beta n^{\alpha}$ for all $n\ge2$, then $t^{(2)}(n)\ge \beta(n/4)^\alpha$ for all $n\ge2$.
\end{prop}
\noindent It follows that  $t^{(2)}(n)=\Omega(\tau(n))$, but,  in fact, we have $t^{(2)}(n)=\Theta(\tau(n))$ (the upper bound is trivial -- see \cite[Section 3.1]{DGR-match} or \cite[Section 3]{DGR-Latin}).

Here we generalize Proposition \ref{polyn} for all values of $r$.
To this end, we first show a (doubly) iterative lower bound on $t^{(r)}(n)$ which also depends on $\tau(n)$. Set $\ell_2=n^{1/3}$ and, for every $r\ge3$, set $\ell_r(n)=\frac12n^{1/(2^r-1)}$.

\begin{lemma}\label{double_ind}
	For all $n\ge1$ and $r\ge2$,
	$t^{(r)}(n)\ge\min\{t_0,t_1,t_2\}$, where
	$$t_0=2t^{(r)}(n/3),\qquad t_1=\tau\left(\ell_{r-1}(n/6r)\right) ,$$
	and
	$$t_2=\min_{2\le p\le r-2}\max\left\{t^{(r-p)}\left(\ell_{p}(n/6r)\right), t^{(p)}\left(\ell_{r-p}(n/6r)\right)\right\}.$$
\end{lemma}
\noindent For $r\le3$, $t_2$ is not defined, or, to the same effect, we may set it equal to $n$ in such cases, so that it does not affect the minimum.

From this we may derive recursively lower bounds on $t^{(r)}(n)$ in terms of $n$ and $r$, assuming a lower bound on $\tau(n)$. For $r\ge1$, let $\eta_r=\frac1{2^{r}-1}$, so that now we can write $\ell_r(n)=\frac12n^{\eta_r}$.
\begin{cor}\label{corol}
	For all $3/5\le\alpha\le2/3$ and $\beta>0$, if $\tau(n)\ge\beta n^\alpha$ for all $n\ge 2$, then for all $r\ge2$ and all $n\ge 2$,
	\begin{equation}\label{boundont}
		t^{(r)}(n)\ge\beta_r(n/6r)^{\alpha\eta_{r-1}},
	\end{equation}
	where $\beta_2=\beta$, while for $r\ge3$,
	$$\beta_r=\min\left\{\min_{2\le p\le r-2}\beta_p(12r)^{-\alpha\eta_{p-1}},\; \beta2^{-\alpha}\right\}.$$
In particular, $t^{(3)}(n)\ge\beta_3(n/18)^{\alpha/3}$, where $\beta_3=\beta2^{-\alpha}$.
\end{cor}

For technical reasons we have not pulled out the constant $(6r)^{-\alpha\eta_{r-1}}$ and incorporated it into $\beta_r$ above.
Note also, that just for the sake of unification, for $r=2$ the bound in Corollary \ref{corol} is slightly weaker than the bound in Proposition \ref{polyn}. The case $r=2$ is special also in that we have $t^{(2)}(n)=\Theta(\tau(n))$ (the upper bound is trivial -- see \cite[Section 3.1]{DGR-match}). On the other hand, for $r\ge3$ the lower bound in Corollary \ref{corol} does not seem to be close to the truth.

Applying the above mentioned bound $\tau(n)=\Omega\left( n^{3/5}\right)$, we obtain immediately that $$t^{(r)}(n)=\Omega\left( n^{\frac35\eta_{r-1}}\right)$$holds for all $r\ge2$. This coincides with (\ref{eq:thm:deterministic}) and  proves Theorem \ref{thm:deterministic}. (For $r=2$ it was already deduced in \cite[Corollary 3.5]{DGR-match}.) So, to complete the proof of Theorem \ref{thm:deterministic}, it remains to prove Lemma~\ref{double_ind} and Corollary~\ref{corol}.

\subsection{Unavoidable patterns}

Given $r\ge2$, there are exactly $\tfrac12\binom{2r}r$ ways, called \emph{patterns}, in which a pair of disjoint edges of order $r$ may intertwine on an ordered vertex set.  
We call them $r$-patterns if the order $r$ is to be emphasized. For instance, using convenient letter notation, there are just three patterns for $r=2$, namely $AABB,ABBA,ABAB$ and ten for $r=3$: $AAABBB$, $AABABB$, $AABBBA$, $AABBAB$, $ABBBAA$, $ABBAAB$, $ABBABA$,  $ABAABB$, $ABABBA$, $ABABAB$. For a pattern $P$, a \emph{$P$-clique} is defined as a matching whose all pairs of edges form pattern $P$. For example, with $P=ABAABB$, a $P$-clique is a 3- matching with the  structure
$A_1\cdots A_n\; A_1A_1\cdots A_nA_n.$
Let $L_P(M)$ be the size of the largest $P$-clique in a matching $M$, $L(M)=\max_PL_P(M)$, and $L_r(n)=\min_ML(M)$, where the minimum is taken over all $r$-matchings $M$ of size $n$.

In \cite{DGR-match} (see also \cite{DGR-Latin}) we showed that $L_2(n)= \lfloor n^{1/3}\rfloor$ and used this result (the lower bound) to prove Proposition \ref{polyn}. Very recently, building upon the concepts and results contained in \cite{DGR-rM} and in \cite{AJKS},
 Sauerman and Zakharov  proved in \cite{SZ}  the following lower bound for every $r$. Recall our notation from the previous subsection:  $\ell_r(n)=\frac12n^{1/(2^r-1)}$ and $\eta_r=\frac1{2^{r}-1}$

\begin{thm}[\cite{SZ}]\label{Lisa}
For all $r\ge2$ and all $n\ge1$, we have $L_r(n)\ge\ell_r(n)$.
\end{thm}
\noindent  In fact, in \cite{SZ} there is a better constant than $\tfrac12$ in front of $n^{\eta_r}$ which depends on $r$ and goes to 1 as $r\to\infty$. For our purposes this has very little effect and therefore we stick to the weaker but easier to handle $1/2$.

\subsection{Proof of Lemma \ref{double_ind}}

Let $M=M^{(r)}(n)$ be an ordered $r$-matching on $[rn]=H_1\cup H_2$, where $H_1=\{\lfloor rn/2\rfloor\}$ and $H_2=[rn]\setminus H_1$ are the first and second ``half'' of the vertex set. Further, for $p=0,1,2,\dots, r-2,r-1,r$, let
$$n_p=|\{e\in M\;: |e\cap H_1|=p\}.$$
 Note that $\sum_{p=0}^rn_p=n$, while $\sum_{p=0}^rpn_p=\lfloor rn/2\rfloor$ and, by symmetry, $\sum_{p=0}^r(r-p)n_p=\lceil rn/2\rceil$. The latter identities imply that $\max\{n_r,n_0\}\le n/2$. We consider three cases (more precisely, two cases one of which splits further into two subcases)  with respect to the values of $n_p$. For two edges $e,f\in M$ we write $e<f$ whenever the leftmost vertex of $e$ is to the left of the leftmost vertex of $f$, i.e., $\min e<\min f$.

\medskip

{\bf Case 1: $\min\{n_0,n_r\}\ge n/3$.} Let $M_r$ and $M_0$ be the sub-matchings of $M$ consisting of the edges contained in, respectively, $H_1$ and $H_2$. Then, for $i=0,r$, we have $|M_i|\ge n/3$, so $t^{(r)}(M_i)\ge t^{(r)}(n/3)$. Thus, by concatenation, $t^{(r)}(M)\ge2 t^{(r)}(n/3)$.

\medskip

{\bf Case 2: $\min\{n_0,n_r\}\le n/3$.} In this case, $n_0+n_r\le n/2+n/3=5n/6$, and so
$$\sum_{p=1}^{r-1}n_p\ge n/6.$$ Thus, we may consider two subcases (for simplicity we compromise $r-1$ to $r$ in all denominators).

\smallskip

{\bf Subcase 2a: $\max\{n_1,n_{r-1}\}\ge n/(6r)$.} W.l.o.g. we assume that $n_{r-1}\ge n/(6r)$. Let $M_{r-1}\subset M$ consist of all $n_{r-1}$ edges of $M$ which intersect $H_1$ in exactly $r-1$ vertices (and thus they intersect $H_2$ in just one vertex). For each $e\in M_{r-1}$ we write $e=e^*\cup\{v_e\}$, where $e^*\subset H_1$ and $v_e\in H_2$. Let $M^*=\{e^*\;: e\in M_{r-1}\}$. Note that $|M^*|=|M_{r-1}|=n_{r-1}$.
By applying Theorem \ref{Lisa} to the $(r-1)$-matching $M^*$ we conclude that for some $(r-1)$-pattern $P$ there is in $M^*$ a $P$-clique $M^*_P=\{e^*_1<\cdots<e^*_m\}$ of size $m=L_{r-1}(n_{r-1})$. Set $v_i:=v_{e_i}$, for convenience.

Let $\pi=v_1,\ldots, v_m$ be the permutation of the right ends of the edges of $M^*_P$ and let $\pi'=v_{i_1},\ldots, v_{i_t}$ and $\pi''=v_{j_1},\ldots, v_{j_t}$ be the longest twins in $\pi$. Then, we claim that $M'=\{e_{i_1}<\cdots<e_{i_t}\}$ and $M''=\{e_{j_1}<\cdots<e_{j_t}\}$ are twins in $M$. Indeed, for $1\le g<h\le t$, consider two pairs of edges $(e_{i_g},e_{i_h})\in M'$ and $(e_{j_g},e_{j_h})\in M''$. The first $r-1$ vertices of $e_{i_g}$ and $e_{i_h}$, i.e., $e^*_{i_g}$ and $e^*_{i_h}$, as well, as $e^*_{j_g}$ and $e^*_{j_h}$ form the same pattern $P$ (as they all belong to the clique $M^*_P$). Moreover, the pairs of rightmost vertices, respectively,  $v_{i_g}$, $v_{i_h}$  and $v_{j_g}$, $v_{j_h}$, as being at the same positions in the twins $\pi'$ and $\pi''$, are in the same relation:  ($v_{i_g}<v_{i_h}$  and $v_{j_g}<v_{j_h}$) or ($v_{i_g}>v_{i_h}$  and $v_{j_g}>v_{j_h}$).

As an example, consider a special case when $r=5$, $p=4$, and $P=ABBBAABA$. Let $e_A,e_B,e_C$, $e_D$ be four edges whose first four vertices form mutually pattern $P$, while the last vertices satisfy $v_A<v_B$ and $v_C<v_D$. Then $e_A$ and $e_B$ form the pattern $ABBBAABA|AB$, while $e_C$ and $e_D$ - pattern $CDDDCCDC|CD$ which is the very same pattern indeed.

This shows that $t^{(r)}(M)\ge \tau\left(L_{r-1}(n_{r-1})\right)\ge \tau\left(L_{r-1}(n/6r)\right)$.

\smallskip

{\bf Subcase 2b: $\max_{2\le p\le r-2}n_p\ge n/(6r)$.} Let, for some $2\le p\le r-2$,\; $n_p\ge n/(6r)$. Let  $M_{p}\subset M$ consist of all $n_{p}$ edges of $M$ which intersect $H_1$ in exactly $p$ vertices (and thus they intersect $H_2$ in $r-p$ vertices). For each $e\in M_{p}$ we write $e=e^*\cup e^{**}$, where $e^*=e\cap H_1$ and $e^{**}=e\cap H_2$. Let $M^*=\{e^*\;: e\in M_{p}\}$. Note that $|M^*|=|M_{p}|=n_{p}$.
By applying Theorem \ref{Lisa} to the $p$-matching $M^*$ we conclude that for some $p$-pattern $P$ there is in $M^*$ a $P$-clique $M^*_P=\{e^*_1<\cdots<e^*_m\}$ of size $m\ge L_p(n_p)$.

Let $M^{**}=\{e^{**}_i:\; i=1,\dots,m\}$ be the $(r-p)$-matching in $H_2$ consisting of the remainders of the edges in $M^*_P$. Further, let $M^{**'}=\{e^{**}_{i_1}<\cdots<e^{**}_{i_t}\}$ and $M^{**''}=\{e^{**}_{j_1}<\cdots<e^{**}_{j_t}\}$ be the largest twins in $M^{**}$. Then, we claim that $M'=\{e_{i_1}<\cdots<e_{i_t}\}$ and $M''=\{e_{j_1}<\cdots<e_{j_t}\}$ are twins in $M$. Indeed, for $1\le g<h\le t$, consider two pairs of edges $(e_{i_g},e_{i_h})\in M'$ and $(e_{j_g},e_{j_h})\in M''$. The $p$-long prefixes of $e_{i_g}$ and $e_{i_h}$, i.e. $e^*_{j_g}$ and $e^*_{j_h})$ form pattern $P$, as the prefixes  $e^*_{j_g}$ and $e^*_{j_h}$ do. Moreover, the $(r-p)$-element suffixes $e^{**}_{i_g}$ and $e^{**}_{i_h}$ form the same $(r-p)$-pattern $Q$ as $e^{**}_{j_g}$ and $e^{**}_{j_h}$ do. So, the pairs  $(e_{i_g},e_{i_h})$ and $(e_{j_g},e_{j_h})$ form the same $r$-pattern $R$, which proves that $M'$ and $M''$ are twins.

To illustrate this part of the proof, consider a special case when $r=5$ and $p=3$. Let $P=AABBBA$ be a collectible 3-pattern and let $Q=ABBA$. If four edges, $e_A,e_B,e_C$, $e_D$  each have 3 vertices in $H_1$ which mutually form pattern $P$ and, moreover, the 2-vertex remainders of $e_A$ and $e_B$ form in $H_2$ pattern $Q$, and the same  holds for the 2-vertex remainders of $e_C$ and $e_D$, then edges $e_A$ and $e_B$ from overall pattern $AABBBA ABBA$, while edges $e_C$ and $e_D$ form pattern $CCDDDCCDDC$, which is the very same pattern (we call it $R$ in our proof).

Hence,
$$t^{(r)}(M)\ge t^{(r-p)}(m)\ge t^{(r-p)}\left(L_p(n_p)\right).$$
We could repeat the entire argument in Subcase 2b with the roles of $p$ and $r-p$ swapped yielding the $\max$ in the definition of $t_2$. Finally, we have to minimize the obtained bound over all $2\le p\le r-2$ and over all three subcases. \qed

\subsection{Proof of Corollary \ref{corol}}

The proof is by double induction on $r\ge2$ and $n\ge 2$. The case $r=2$ was proved in \cite{DGR-match}, see Proposition \ref{polyn} above. Indeed, by Proposition \ref{polyn} we have
$$t^{(2)}(n)\ge \beta(n/4)^\alpha\ge\beta(n/12)^\alpha=\beta_2(n/12)^\alpha.$$
 Moreover, for $r\ge3$, if $2\le n\le6r(1/\beta_r)^{\frac1{\alpha\eta_{r-1}}}$, then the desired bound becomes $t^{(r)}(n)\ge1$ which is trivially true.
Thus, assume that $r\ge3$, $n\ge6r(1/\beta_r)^{\frac1{\alpha\eta_{r-1}}}$, equivalently,
\begin{equation}\label{assonn}
\beta_r(n/6r)^{\alpha\eta_{r-1}}\ge1
\end{equation}
and that \eqref{boundont} holds for all $2\le r'<r$ and $2\le n'<n$.

Note that the assumption $\tau(n)\ge\beta n^\alpha$ for all $n\ge 2$ implies, by taking $n=2$, that $1\ge \beta 2^{\alpha}$, or $\beta\le 2^{-\alpha}$. Thus, whenever \eqref{assonn} holds, recalling that $\beta_r\le2^{-\alpha}\beta$, it follows that
$$\tfrac12(n/6r)^{1/\eta_{r-1}}\ge\tfrac12\beta_r^{-1/\alpha}\ge\beta^{-1/\alpha}\ge2.$$
This, in turn, implies that $n/6r>1$ and, consequently, that in Lemma \ref{double_ind}  all arguments appearing in the functions defining $t_0$, $t_1$, and $t_2$  are larger or equal to 2 (and, obviously, less than $n$) allowing to apply  the induction assumptions.

By Lemma \ref{double_ind}, we know that $t^{(r)}(n)\ge\min\{t_0,t_1,t_2\}$. So, consider three cases.

\medskip

{\bf Case 0: $t^{(r)}(n)\ge t_0$.} As $n/3\ge2$, by induction's assumption we have
$$t^{(r)}(n)\ge t_0=2t^{(r)}(n/3)\ge2\beta_r(n/18r)^{\alpha\eta_{r-1}}\ge \beta_r(n/6r)^{\alpha\eta_{r-1}},$$
the last inequality equivalent to $2^{1/\eta_{r-1}}=2^{2^{r-1}-1}\ge3^\alpha$, which is true for $r\ge3$ (recall that $\alpha\le1$).

\medskip

{\bf Case 1: $t^{(r)}(n)\ge t_1$.} In this case, by the assumption,
$$t^{(r)}(n)\ge t_1=\tau\left(\tfrac12(n/6r)^{\eta_{r-1}}\right)\ge \beta2^{-\alpha}(n/6r)^{\alpha\eta_{r-1}}\ge\beta_r(n/6r)^{\alpha\eta_{r-1}},$$
since $\beta_r\le2^{-\alpha}\beta$.

\medskip

{\bf Case 2: $t^{(r)}(n)\ge t_2$.} By induction's assumption,

\begin{align*}
t^{(r)}(n)\ge t_2 &=\min_{2\le p\le r-2}\max\left\{t^{(r-p)}\left(\tfrac12(n/6r)^{\eta_p}\right),\; t^{(p)}\left(\tfrac12(n/6r)^{\eta_{r-p}}\right)\right\}\\
&\ge \min_{2\le p\le r-2}\max\left\{\beta_{r-p}\left(\frac{(n/6r)^{\eta_p}}{12(r-p)}\right)^{\alpha\eta_{r-p-1}},\; \beta_p\left(\frac{(n/6r)^{\eta_{r-p}}}{12p}\right)^{\alpha\eta_{p-1}}\right\}\\
&\ge\min_{2\le p\le r-2}\max\left\{\beta_{r-p}(12r)^{-\alpha\eta_{r-p-1}},\;\beta_p(12r)^{-\alpha\eta_{p-1}}\right\}(n/6r)^{\alpha\eta_{r-1}}\\
&\ge\beta_r(n/6r)^{\alpha\eta_{r-1}},
\end{align*}
by choosing $\beta_r\le\min_{2\le p\le r-2}\beta_p(12r)^{-\alpha\eta_{p-1}}.$ (Above we used the facts that $n/6r>1$ and $\eta_p\eta_{r-p-1}>\eta_{r-1}$, both with a big margin.)
\qed

\section{Concluding remarks}

Let us conclude the paper with some problems for future considerations. Firstly, it is natural to speculate on the true asymptotic order of the function $t^{(r)}(n)$. Based on Theorem~\ref{thm:random} and the former results around Gawron's conjecture \cite{Gawron} on twins in permutations (see also \cite{DGR}), we dare to state the following.

\begin{conj}\label{conj main}
	For every $r\geqslant 2$, $t^{(r)}(n)=\Theta\left(n^{\frac{2}{r+1}}\right)$.
\end{conj}
Notice that the case $r=2$ of this statement is the original Gawron's conjecture. It seems that even increasing the exponent of $n$ in the lower bound (\ref{eq:thm:deterministic}) to the inverse of any polynomial in $r$ will be quite a challenge.

One could also consider a generalization to $t$-tuplets, that is, $t$-tuples of vertex disjoint order-isomorphic sub-matchings of a given $r$-matching.
In fact, the proof techniques from Section~\ref{rm} yield that  this generalized parameter is a.a.s.\ $\Theta(n^{t/(r(t-1)+1)})$ (for $t=2$ we thus recover Theorem~\ref{thm:random}). A highly technical analog of Theorem \ref{thm:deterministic} could be proved as well.

It seems also natural to study the problem of the largest twins in more general classes of ordered graphs ($r=2$) or hypergraphs ($r\ge3$). By \emph{twins} in an ordered graph $G$ we mean a pair of edge-disjoint order-isomorphic subgraphs of $G$. Let $t(G)$ denote the maximum size of  twins in $G$, and let $t(m)$ be the minimum of $t(G)$ over all ordered graphs with $m$ edges. What can be said about the function $t(m)$? By the results for ordered matchings we know only that $t(m)=O(m^{\frac{2}{3}})$, but is it  optimal? It is worth mentioning that the analogous problem for unordered graphs was  solved by Lee, Loh, and Sudakov \cite{LeeLohSudakov} who proved that the corresponding function is  $\Theta\left((m\log m)^{\frac{2}{3}}\right)$.

Finally, one may also investigate  the size of twins in general words over finite alphabets. Therein, twins are defined as  pairs of identical subsequences occupying disjoint sets of positions. Actually,  our motivation to study twins in permutations and ordered matchings has been ignited by a beautiful result of Axenovich, Person, and Puzynina \cite{AxenovichPP}, stating that every binary word of length $n$ contains twins of size $\frac{1}{2}n-o(n)$. 

Inspired by the word representation of ordered matchings (used also in this paper), one may consider a relaxed variant of twins in words, in which a pair of disjoint  subwords forms \emph{permuted twins} if they are identical up to a permutation of their letters. For instance, in this new setting, the subwords $112142212$ and  $223213323$ of a word over alphabet $\{1,2,3,4\}$, would form permuted twins (under permutation $2341$). How large permuted twins  may one find in every word of length $n$ over a $k$-element alphabet? Will they be much bigger, especially for large $k$, than in the classical case?

\bigskip
\noindent
\textbf{Acknowledgments.}
 We would like to thank an anonymous referee for a careful reading of the manuscript and suggesting a number of editorial improvements.

\end{document}